\documentclass[a4paper,12pt]{article}
\usepackage[top=2.5cm,bottom=2.5cm,left=2.5cm,right=2.5cm]{geometry}
\usepackage{cite, amsmath, amssymb}
\usepackage[margin=1cm,%
font=small,%
format=hang,%
labelsep=period,%
labelfont=bf]{caption}
\usepackage{amssymb,amsfonts,amsmath,latexsym,epsf,tikz,url}
\usepackage[usenames,dvipsnames]{pstricks}
\usepackage{pstricks-add}
\usepackage{epsfig}
\usepackage{pst-grad} 
\usepackage{pst-plot} 
\usepackage[space]{grffile} 
\usepackage{etoolbox} 
\makeatletter 
\patchcmd\Gread@eps{\@inputcheck#1 }{\@inputcheck"#1"\relax}{}{}
\makeatother

\newtheorem{theorem}{Theorem}[section]

\newtheorem{corollary}[theorem]{Corollary}
\newtheorem{lemma}[theorem]{Lemma}

\newtheorem{observation}[theorem]{Observation}

\newcommand{\qed}{\hfill $\square$\medskip}

\begin{document}

\title{Pebbling Number of Polymers}
\author{Fatemeh Aghaei \and 
Saeid Alikhani\footnote{Corresponding author}
}

\date{\today}

\maketitle

\begin{center}
Department of Mathematical Science, Yazd University, 89195-741, Yazd, Iran\\
{\tt \tt aghaeefatemeh29@gmail.com, alikhani@yazd.ac.ir}
\end{center}

\begin{center}
(Received February 28, 2021) 
\end{center}

\begin{abstract}
  Let $G=(V,E)$ be a simple graph. 
  A function $f:V\rightarrow \mathbb{N}\cup \{0\}$ is called a configuration of pebbles on the vertices of $G$  and the quantity $\vert f\vert=\sum_{u\in V}f(u)$
  is called the weight of $f$ which is  just the total number of pebbles assigned to vertices. A pebbling step from a vertex $u$ to one of its
  neighbors $v$  reduces $f(u)$ by two and increases $f(v)$  by one. A pebbling configuration $f$ is said to be solvable if for every vertex $ v $, there exists a sequence (possibly empty) of pebbling moves that results in a pebble on $v$. The pebbling number $ \pi(G) $ equals the minimum number $ k $ such that every pebbling configuration $ f $ with $ \vert f\vert = k $ is solvable. 
  Let $ G $ be a connected graph constructed from pairwise disjoint connected graphs $ G_1,...,G_k $ by selecting a vertex of $ G_1 $, a vertex of $ G_2 $, and identifying these two vertices. Then continue in this manner inductively. We say that $ G $ is a polymer graph, obtained by point-attaching from monomer units $ G_1,...,G_k $. 
  In this paper, we study the pebbling number of some polymers.
\end{abstract}

\baselineskip=0.30in

\section{Introduction and definitions}
Let $G=(V,E)$ be a simple graph of order $n$. 
A function $f:V\rightarrow \mathbb{N}\cup \{0\}$ is called a configuration of pebbles on the vertices of $G$  and the quantity $\vert f\vert=\sum_{u\in V}f(u)$
is called the size of $f$ which is  just the total number of pebbles assigned to vertices. A pebbling step from a vertex $u$ to one of its neighbors $v$  reduces $f(u)$ by two and increases $f(v)$  by one. A pebbling configuration $f$ is said to be solvable if for every vertex $ v $, there exists a sequence (possibly empty) of pebbling moves that results in a pebble on $v$. The pebbling number $ \pi(G) $ equals the minimum number $ k $ such that every pebbling configuration $ f $ with $ \vert f\vert = k $ is solvable.

Given a specified target vertex $ r $ we say that $ f $ is $ t $-fold $ r $-solvable if some sequence of pebbling steps places $ t $ pebbles on $ r $. The $ t $-fold pebbling number of $ G $ is defined to be $ \pi_{t}(G)=\max _{r\in V(G)}\pi_{t}(G,r) $. When $ t=1 $, we have $ \pi(G) $.

The configuration with a single pebble on every vertex except the target shows that $\pi(G)\geq n$, while the configuration with $2^{ecc(r)}-1$  pebbles on the farthest vertex from $r$, and no
pebbles elsewhere, shows that $\pi(G)\geq 2 ^{{\rm diam}(G)}$ when $r$ is
chosen to have $ecc(r)=diam(G)$. Because the $ r $-solvability of a configuration is not destroyed by adding edges, for every root vertex $ r $  we have  $ \pi(H,r)\geqslant\pi(G,r) $ whenever $ H $ is a connected spanning subgraph of $ G $. It is called the subgraph bound.

As usual let $Q^d$ be the $d$-dimensional hypercube  as the graph on all binary $d$-tuples,
pairs of which that differ in exactly one coordinate are
joined by an edge. Chung \cite{Chung} proved that $\pi(Q^d)=2^d$. Graphs $G$ like
$Q^d$  which have $\pi(G)=|V(G)|$   to be known as Class $0$.
The terminology comes from a lovely theorem of Pachter,
Snevily, and Voxman \cite{Pachter}, which states that if $diam(G)=2$, 
then $\pi(G)\leqslant n+1$. 
Therefore there are two classes of diameter two graphs, Class $0$ and
Class $1$. The
Class $0$ graphs are $2$-connected.

Pachter et al. in \cite{Pachter}, defined the optimal pebbling number $ \pi^{*}(G) $ to be the minimum weight of a solvable pebbling configuration of $G$. A solvable pebbling configuration of $G$ with weight $ \pi^{*}(G) $ is called a $ \pi^{*}$-configuration. The decision problem associated with computing the optimal pebbling number was shown to be NP-Complete in \cite{Milans}.

M. Chellali et al. in \cite{Chellali} introduced a generalization of the optimal pebbling number, that a pebbling configuration $  f $ is a $ t $-restricted pebbling configuration (abbreviated $ t RPC $) if $ f(v) \leqslant t $ for all $ v \in V $. They defined the $ t $-restricted optimal pebbling number $ \pi^{*}_{t}(G) $ to be the minimum weight of a solvable $  t RPC $ on $ G $. If $f$ is a solvable $ t RPC $ on $G$ with $ \vert f\vert=\pi^{*}_{t}(G) $, then $f$ is called a $ \pi^{*}_{t} $-configuration of $ G $. We note that the limit of $ t  $ pebbles per vertex applies only to the initial configuration. That is, a pebbling move may place more than $  t $ pebbles on a vertex.

Let $ G $ be a connected graph constructed from pairwise disjoint connected graphs $ G_1,...,G_k $ by selecting a vertex of $ G_1 $, a vertex of $ G_2 $, and identifying these two vertices. Then continue in this manner inductively. We say that $ G $ is a polymer graph, obtained by point-attaching from monomer units $ G_1,...,G_k $. Such graphs can be decomposed into subgraphs that we call monomer units. Cacti are some particular cases of these graphs. Sombor index of these kinds of graphs has studied in \cite{Alikhani4}. Also the distinguishing labeling and distinguishing number  of these kinds of graphs studied in \cite{IJMC}. 
Fluorescent labeling of biocompatible block copolymers studied in \cite{Mater}.

\medskip

In this paper, we consider the pebbling number of polymer graphs. 
In Section 2, we consider some special cases of polymers which is called the generalized friendship graphs and obtain their pebbling numbers.
In Section 3,  the pebbling number  of some cactus chains which are polymer graphs and that are of importance in chemistry are computed.  We extend the  results of Sections 2 and 3, in order to study the pebbling number  of polymer graphs.

\section{Pebbling number of some generalized friendship graphs}
The friendship graph $ F_{n,3} $ can be constructed by joining $ n $ copies of the cycle graph $ C_{3} $ with a common vertex $ v $ (see Figure \ref{f4}). The generalized friendship graph $ F_{n,m} $ is a collection of $ n $ cycles (all of order $ m $), meeting at a common vertex (see Figure  \ref{f5}). The generalized friendship graphs are cacti. In \cite{Alikhani1} we proved $ \pi(F_{n,3})=2n+2 $, i.e., friendship graphs are class $1$.
By putting  two pebbles on the common vertex $ v $ of the friendship graph $ F_{n,3} $, we obtain the following observation: 

\begin{observation}
	$ \pi^{*}(F_{n,3})=\pi^{*}_{2}(F_{n,3})=2. $
\end{observation}

\begin{figure}[ht]
	\centering
	\includegraphics[scale=.6]{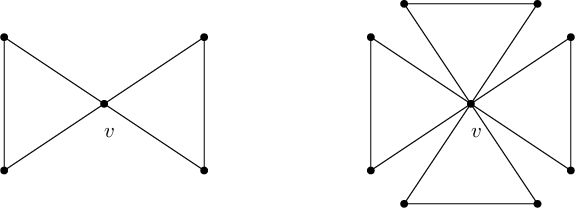}
	\caption{The friendship graphs $ F_{2,3} $ and $ F_{4,3} $}\label{f4}
\end{figure}

Pachter, Snevily and Voxman \cite{Pachter} have computed  the pebbling numbers of cycles. Snevily and Foster \cite{Snevily} gave an upper bound for the $ t $-pebbling numbers of odd cycles. This bound was shown to be exact in \cite{Herscovici}, which gave the $t$-pebbling number of even cycles as well. These numbers are given in Theorem \ref{2.2}.

\begin{theorem}{\rm\cite{Herscovici, Pachter, Snevily}}\label{2.2}
	The $ t $-pebbling number of the cycles $ C_{2n} $ and $ C_{2n+1} $ are
	\begin{equation*}
	\pi_{t}(C_{2n})=t2^{n} 
	~~~and~~\pi_{t}(C_{2n+1})=\dfrac{2^{n+2}-(-1)^{n}}{3}+2^{n}(t-1). 
	\end{equation*}
	In particular, 
	$ \pi(C_{2n})=2^{n} $ and $ \pi(C_{2n+1})=\dfrac{2^{n+2}-(-1)^{n}}{3}$.
\end{theorem}

We present  another proof for the $ t $-pebbling numbers of even cycles. Note that $\pi(C_{2n})=2^n$. 
\begin{theorem} 
	For every number $n\geq 2$,	$\pi_{t}(C_{2n})=t2^{n}$.
\end{theorem} 
\begin{proof}
	Consider the graph $H$ as $ t $ copies of a graph $ G $ and a configuration $ f_{i} $ of size $ \vert f_{i}\vert=\pi(G) $ on the copy $ G_{i}=G $ ($ i=1,...,t $). So,  the configurations $ f_{i} $ are solvable. Let $ F(v)=\sum f_{i}(v) $ be a configuration  on the graph $ H $ then by putting at least $ t $ pebbles on any vertex of $ G $, we have $ \pi_{t}(G)\leqslant t\pi(G) $. So an upper bound for the $ t $-pebbling numbers of even cycles is $ \pi_{t}(C_{2n})\leqslant t2^{n} $. Suppose that $V(C_{2n})=\{v_1,...,v_{2n}\}$ such that the vertex with label $v_i$ is adjacent to vertex with label $v_{i+1}$.   
	If we place $ t2^{n}-1 $ pebbles on the vertex $v_i$, then we cannot move $ t $ pebbles to the vertex $v_{n+i}$ which consider as the target vertex (note that two vertices $v_i$ and $v_{n+i}$ are opposite, if the even cycles considered as regular polygon).  Therefore $ \pi_{t}(C_{2n})=t2^{n} $.\qed
\end{proof}

Now by the $ t $-pebbling numbers of even cycles, we obtain the pebbling number of the generalized friendship graph $F_{n,2k}$. 

\begin{theorem}\label{*}
	$ \pi(F_{n,2k})=2^{2k}+(2^{k}-1)(n-2) $, for $ n,k\geqslant 2 $. 
\end{theorem}
\begin{proof}
	Let $ F_{n,2k} $ is constructed by joining $ n $ copies of the cycle graph
	$ C_{2k} $ with a common vertex $ v $. We have two cases:
	\begin{itemize}
		\item[(1)]
		The target vertex $ r $ is the common vertex in the graph $ F_{n,2k} $. In this case,  by pigeonhole principle we have $  \pi(F_{n,2k},r)=n(2^{k}-1)+1 $.
		\item[(2)]
		The target vertex $ r $ is not the common vertex $ v $. Therefore, we need at least $ 2^{k} $ pebbles to solve the target vertex in the cycle $ C_{2k} $. If we put $ 2^{k}-1 $ pebbles on the common vertex, then the opposite vertex of target $ r $ cannot solve. By Theorem \ref{*}, $ \pi_{2^{k}}(C_{2k})=2^{2k} $. So $ f $ is the largest $ r $-unsolvable configuration of maximum size $ (2^{2k}-1) +(2^{k}-1)(n-2) $ that gives $ 2^{2k}-1 $ pebbles to unique cycle and $ 2^{k}-1 $ pebbles to other cycles. 
		Therefore $ \pi(F_{n,2k})=\max_{r}\pi(F_{n,2k},r)=2^{2k}+(2^{k}-1)(n-2) $.\qed
	\end{itemize}
\end{proof}

We need the following theorem to obtain more result on the pebbling number, the optimal pebbling number and the $2$-restricted optimal pebbling number   of $F_{n,4}$.

\begin{theorem}\rm{\cite{Alikhani3}}\label{**}
	Let $G$ be a nontrivial connected graph. Then
	$ \pi^{*}_{2}(G)=5 $ if and only if $ \gamma_{t}(G)\geqslant 4 $ and while for any $ u, v\in V(G) $, $  \lbrace u, v \rbrace\cup(N(u)\cap N(v))\ $ cannot dominate $ G $, one of the following conditions happens:
	\begin{itemize}
		\item[(i)]
		There are three vertices $ u, v $ and $ w\in N(u)\cap N(v) $ such that $ \lbrace u, v, w\rbrace\cup(N(u)\cap N(v))\cup(N(u)\cap N(w))\cup(N(v)\cap N(w))$ dominates $ G $.
		\item[(ii)]
		There are  $ u , v $ and $ w\in N(v) $ such that $ \lbrace u, v, w\rbrace\cup(N(u)\cap N(v))\cup(N(u)\cap N(w))$ dominates $ G $.
		\item[(iii)]
		There are  $ u , v $ and $ w\in N(N(u)\cap N(v)) $ such that $ \lbrace u, v, w\rbrace\cup(N(u)\cap N(v))$ dominates $ G $.
	\end{itemize}  
\end{theorem}
\begin{theorem}
	For $ n\geqslant 2 $, we have:
	\begin{itemize}
		\item[(i)]
		$\pi(F_{n,4})=3n +10.$
		\item[(ii)]
		$ \pi^{*}(F_{n,4})=4. $
		\item[(iii)]
		$
		\pi^{*}_{2}(F_{n,4})= 
		\left\{
		\begin{array}{ll}
		4, & \mbox{if}\quad{n=2};\\ [10pt]
		5, & \mbox{if}\quad{n=3};\\[10pt]
		6, & \mbox{if}\quad{n\geq 4.}
		\end{array}
		\right.
		$
	\end{itemize}
\end{theorem}
\begin{proof}
	\begin{itemize}
		\item[(i)]
		The result follows from Theorem \ref{*}.
		\item[(ii)]
		Let $ f $ be a configuration that put  $ 4 $ pebbles on the common vertex of the graph $ F_{n,4} $. Therefore $ \pi^{*}(F_{n,4})\leqslant 4 $. Since $ \pi^*(C_{4})=3 $ and there is no a solvable configuration of size $ 3 $, so $ \pi^{*}(F_{n,4})=4$.  
		\item[iii)]
		Since for any graph $G$, $ \pi^{*}(G)\leqslant\pi^{*}_{2}(G)$, so $ \pi^{*}_2(F_{n,4})\geqslant 4 $. Now consider the solvable configuration $ f $ of size $ 4 $ for the  graph $ F_{2,4} $  as shown in  Figure \ref{f5}. So $ \pi^{*}_{2}(F_{2,4})=4 $.
		For $ n=3 $, the result follows from Theorem \ref{**}. For $ n>3 $, we have $ \pi^{*}_{2}(F_{n,4})>5 $ by Theorem \ref{**}. Consider the solvable configuration of size $ 6 $ that has shown  in Figure \ref{f5}, so $ \pi^{*}_{2}(F_{n,4})=6 $, $ (n>3) $. \qed
	\end{itemize}
\end{proof}

\begin{figure}[ht]
	\centering
	\includegraphics[scale=.6]{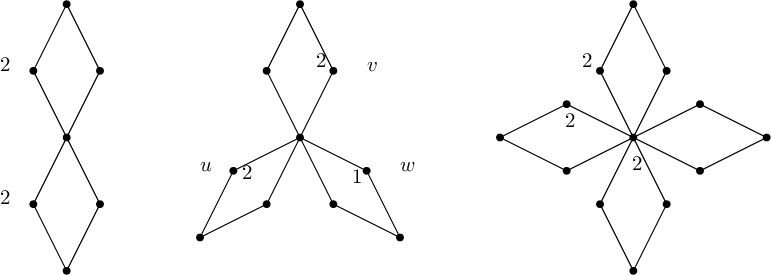}
	\caption{The friendship graphs $ F_{2,4} $, $ F_{3,4} $ and $ F_{4,4} $}\label{f5}
\end{figure}

We close this section by the following questions: 

\medskip
\noindent {\bf Question 1:} Compute the  optimal pebbling and $ 2 $-restricted optimal pebbling number of graph $ F_{n,2k} $, for $ k>2 $.

\medskip
\noindent {\bf Question 2:} Compute  $ \pi(F_{n,2k+1}), \pi^*(F_{n,2k+1})$, and $\pi^*_2(F_{n,2k+1})$.

\section{Pebbling number of some cactus chains}

In this section, we obtain the pebbling number  of families of graphs that are of importance in chemistry. These graphs are cactus graphs. 
A cactus is a connected graph in which any two simple cycles have at most one vertex in common. Equivalently, it is a connected graph in which every edge belongs to at most one simple cycle, or in which every block (maximal subgraph without a cut-vertex) is an edge or a cycle. If all blocks of a cactus $ G $ are cycles of the same size $ i $, the cactus is $ i $-uniform. A triangular cactus is a graph whose blocks are triangles, i.e., a $ 3 $-uniform cactus. A vertex shared by two or more triangles is called a cut-vertex. If each triangle of a triangular cactus $ G $ has at most two cut-vertices, and each cut-vertex is shared by exactly two triangles, we say that $ G $ is a chain triangular cactus. By replacing triangles in this definition by cycles of length $ 4 $ we obtain cacti whose every block is $ C_{4} $ that are called square cacti. Note that the internal squares may differ in the way they connect to their neighbors. If their cut-vertices are adjacent, we say that such a square is an ortho-square; if the cut-vertices are not adjacent, we call the square a para-square \cite{Alikhani2}.

We need recall the maximum $r$-path partition of a tree. 
Suppose that $ \mathcal{P}=\lbrace P^{1},...,P^{m}\rbrace $ is a path partition of a tree $ T $, with each $ P^{i} $ having length $ \mathit{l}_{i} $ written in non-increasing order. We say that $ \mathcal{P} $ is an $ r $-path partition if $ r $ is an endpoint of $ P^{1} $ and any other $ P^{i} $ that contains it, and that $ \mathcal{P} $ majorizes another $ r $-path partition $ \mathcal{P}' $ if there is some $ j $ such that $ \mathit{l}_{i}=\mathit{l}'_{i} $ for all $ i<j $ and $ \mathit{l}_{j}>\mathit{l}'_{j} $ . If $ \mathcal{P} $ majorizes every $ r $-path partition, then it is maximum. Chung \rm\cite{Chung} proved the following theorem.

\begin{theorem}{\rm\cite{Chung}}
	If $ \mathit{l}_{1},...,\mathit{l}_{m} $ are the path lengths of a maximum $ r $-path partition of a tree $ T $, then $ \pi_{t}(T,r)={\displaystyle t2^{\mathit{l}_{1}}+ \sum_{i=2}^{m}2^{\mathit{l}_{i}} -m +1}$.
\end{theorem}
\begin{corollary}{\rm\cite{Bunde}}
	If the length list of an optimal path partition of tree $ T $ is $ \mathit{l}_{1},...,\mathit{l}_{m} $, then
	$ \pi(T)={\displaystyle\sum_{i=1}^{m}2^{\mathit{l}_{i}} -m +1} $
\end{corollary}

\begin{figure}[ht]
	\centering
	\includegraphics[scale=.76]{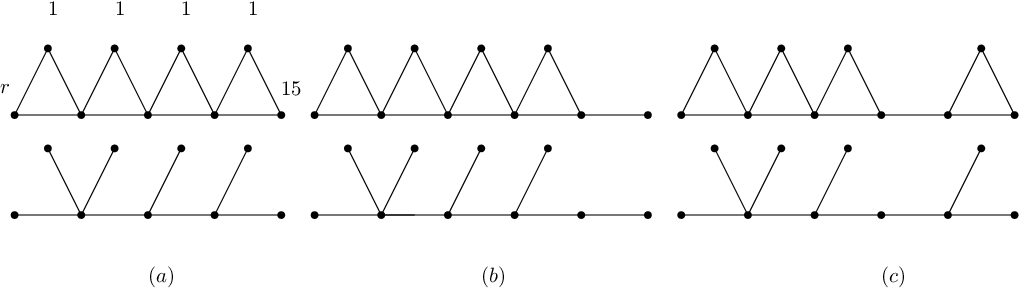}
	\caption{(a) An unsolvable configuration $ f $ of size $ 19 $ and a spanning tree of graph $ T_{4} $, (b) Graph $ T_{4}+e $ and an its spanning tree, (c)  Graph $ T_{4}+e' $ and an its spanning tree.}\label{f1}
\end{figure}

Here, we consider triangular cactus and call the number of triangle, the length of the chain.  Obviously, all chain triangular cacti of the same length are isomorphic. Hence, we denote the chain triangular cactus of length $n$ by $T_n$. The following theorem gives the pebbling number of $T_n$.

\begin{theorem}
	$ \pi(T_{n})=2^{n}+n $.
\end{theorem}
\begin{proof}
	By subgraph bound, we have $ \pi(T_{n})\leqslant\pi(S_{n}) $, where $ S_{n} $ is a spanning tree (with minimum eccentric) of a triangular chain cactus with $ n $ triangles, so $ \pi(T_{n})\leqslant 2^{n}+n  $. According to Figure \ref{f1} $ (a) $, one can consider the unsolvable configuration $ f $ of size $ 2^{n}+n-1 $. Therefore we have the result.\qed
\end{proof}

\begin{corollary}
	$ \pi(T_{n}+e)=2^{n+1}+n $.
\end{corollary}

To obtain more results for square cactus chains we need the following preliminaries. 
If  $ P_{n} $ is the path $ v_{1}v_{2}...v_{n} $ on $ n $ vertices, then it is easy to see that $ \pi(P_{n})=2^{n-1} $.  
Now define $ w $ on $ V(P_{n}) $ by $ w(v_{n-i})=2^{i} $, and extend the weight function to configurations by $ w(f)=\sum_{v\in V}w(v)f(v) $. Then a pebbling step can only preserve or decrease the weight of a configuration. Since the weight of a configuration with a pebble on $ v_{1} $ is at least $ 2^{n-1} $, we see that $ 2^{n-1} $ is a lower bound on every $ v_{1} $-solvable configuration. In fact, induction shows that every $ v_{1} $-unsolvable configuration has weight at most $ 2^{n-1}-1 $, which equals $ \sum_{i=2}^{n}w(v_{i}) $.

Let $ G $ be a graph and $ T $ be a subtree of $ G $ rooted at vertex $ r $, with at least two vertices. For a vertex $v\in V(T)$ let $ v^{+} $ denote the parent of $ v $; i.e. the $T$-neighbor of $ v $ that is one step closer to $ r $ (we also say that $ v $ is a child of $ v^{+} $). The tree $ T $ is called a \textit{strategy} when we associate with it a nonnegative, nonzero weight function $ w $ with the property that $ w(r) = 0 $ and $ w(v^{+}) = 2w(v) $ for every other vertex that is not a neighbor of $ r $ (and $ w(v) = 0 $ for vertices not in $ T $). Let $ \textbf{T} $ be the configuration with $ \textbf{T}(r) = 0, \textbf{T}(v) = 1 $ for all $ v\in V (T) $, and $ \textbf{T}(v) = 0 $ everywhere else.

\begin{lemma}{\rm\cite{Hurlbert} [Weight Function Lemma]}\label{wfl} 
	Let $ T $ be a strategy of $ G $ rooted at $ r $, with associated weight function $ w $. Suppose that $ f $ is an $ r $-unsolvable configuration of pebbles on $ V(G) $. Then $ w(f)\leqslant w(\textbf{T}) $.
\end{lemma}

For a graph $ G $ and root vertex $ r $, let $ \mathbb{T} $ be the set of all $ r $-strategies in $ G $, and denote by $ z^*_{G,r} $ the optimal value of the integer linear optimization problem $ \textbf{P}_{G,r} $:
\begin{center}
	Max. $ \sum_{v\neq r}f(v) $ s.t. $ w(f)\leqslant w(\textbf{T}) $, and $ \textbf{T}\in\mathbb{T} $ with witnessing weight function $ w $.
\end{center}

\begin{corollary}{\rm\cite{Hurlbert}}
	Every graph $ G $ and root $ r $ satisfies $ \pi(G,r)\leqslant  z^*_{G,r} +1$. 
\end{corollary}

Let explain more about Lemma \ref{wfl}.  Suppose that $ T $ is a tree with target vertex $r$ which is  one of its leaves, and  $ ecc_{T}(r)=d $. Define the weight function $ w_T $ by $ w_T(v)=2^{d-i} $, where  $ i=dist_T(v,r) $. The pair $ (T,w_T) $ is called a basic strategy.
In fact, the Weight Function Lemma says that if $ f $ is an $ r $-unsolvable configuration on $ T $, then
\[\sum_{v\in T}w_T(v)f(v)\leqslant\sum_{v\in T}w_T(v). \]
One can extend this to all graphs as follows. Given a target $ r $ in a graph $ G $, consider any tree $ T $ in $ G $ with $ r $  as a leaf. Extend $ w_T $ to all of $ G $ by setting $ w_T(v)=0 $ when $ v\notin T $. So the Weight Function Lemma still holds.
The collection of all inequalities from basic strategies gives rise to an integer optimization problem by maximizing $ z_{G,r}=\sum_{v\in G}f(v) $ (the size of $ f $) over these constraints. If $ z^*_{G,r} $ is the optimum value, then it shows that every $ r $-unsolvable configuration has size at most $ z^*_{G,r} $; in other words, $ \pi(G,r)\leqslant  z^*_{G,r} +1$.

\begin{figure}[ht]
	\centering
	\includegraphics[scale=.9]{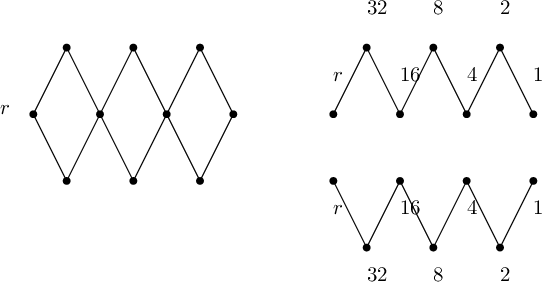}
	\caption{The tree constraints of $ Q_{3}$. }\label{f3}
\end{figure}

\begin{theorem}
	If $ Q_{n} $ is a para-chain cactus with $ n $ squares, then $ \pi(Q_{n},r)=2^{2n}$.
\end{theorem}
\begin{proof}
	The constraint trees have shown in Figure \ref{f3} and the sum of both of them corresponds to 
	\[
	2(2^{2n-1}+2^{2n-2}+...+2^{0})= 2(2^{2n}-1),
	\]
	so  $ \pi(Q_{n},r)\leqslant 2^{2n} $. Since $ \pi(Q_{n},r)\geqslant 2^{ecc_{Q_{n}}(r)} $, we have $ \pi(Q_{n},r)=2^{2n}$.
	\qed
\end{proof}

\begin{corollary}
	$ \pi(Q_{n}+e,r)=2^{2n+1}. $
\end{corollary}

\begin{figure}[ht]
	\centering
	\includegraphics[scale=0.8]{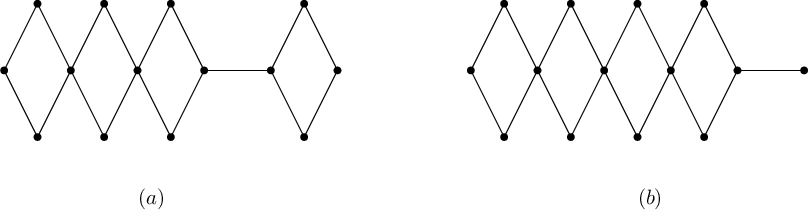}
	\caption{The graphs $ Q_{4}+e $.}\label{f6}
\end{figure}

A \textit{nonbasic strategy} will satisfy the inequality $ w(v^+)\geqslant w(v) $ in place of the equality used in a basic strategy. The following lemma shows that we can use nonbasic strategies in an upper bound certificate since they are conic combinations of a nested family of basic strategies.
\begin{lemma}\rm{\cite{Hurlbert}}
	If $ T $ is a nonbasic strategy for the rooted graph $ (G,r) $, then there exists basic strategies $ T_{1},...,T_{k} $ for $ (G, r) $ and non-negative constants $ c_{1},...,c_{k} $ so that $ T=\sum_{i=1}^{k}c_{i}T_{i} $.
\end{lemma}

\begin{figure}[ht]
	\centering
	\includegraphics[scale=0.85]{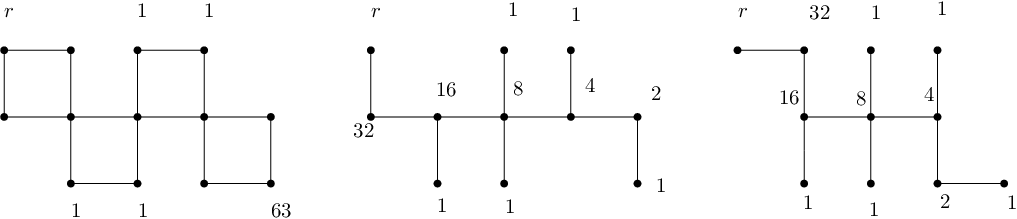}
	\caption{A $ r $-unsolvable configuration of graph $ O_{4} $ and nonbasic strategies.}\label{f7}
\end{figure}

\begin{theorem}
	If  $ O_{n} $ is a ortho-chain square cactus with $ n $ squares, then 
	$$ \pi(O_{n},r)=2^{n+2}+2n-4. $$
\end{theorem}

\begin{proof}
	We consider the nonbasic strategies which have shown in Figure \ref{f7}. Since 
	\[
	2(\sum_{i=0}^{n+1}2^{i})+2(2n-4)= 2(2^{n+2}-1)+2(2n-4),
	\]
	so  $ \pi(O_{n},r)\leqslant 2^{n+2}+2n-4 $. Let $ f $ be the $ r $-unsolvable configuration that has been shown in Figure \ref{f7}. Therefore the result follows. \qed
\end{proof}


\section{Pebbling number of Polymers}
In this section,  we study  the pebbling number  of some graphs  from their monomer units.

Let $ G $ be a connected graph constructed from pairwise disjoint connected graphs $ G_1,...,G_k $ as follows. Select a vertex of $ G_1 $, a vertex of $ G_2 $, and identify these two vertices that for simplicity we will call it a node. Then continue in this manner inductively. Note that the graph $G $ constructed in this way has a tree-like structure, the $ G_i $’s being its building stones (see Figure \ref{f8}) and say that $ G $ is a \textit{polymer} graph, obtained by point-attaching from $ G_1,...,G_k $ and that $ G_i $’s are the \textit{monomer} units of $ G $. A particular case of this construction is the decomposition of a connected graph into blocks \cite{Emeric}.

Let $ r $ be the target vertex in graph $ G $. A pebbling step from $ u $ to $ v $ is $ r $-greedy if $ dist(v,r)<dist(u,r) $. It is $ r $-semigreedy if $dist(v, r) \leqslant dist(u, r)$. Furthermore, A set of pebbling steps is $ r $-(semi) greedy if every one of its steps is $ r $-(semi) greedy, a configuration is $ r $-(semi) greedy it is has an $  r$-(semi)greedy solution, and a graph G is $ r $-(semi) greedy if every configuration of size at least $ \pi(G,r) $ is $ r$- (semi-) greedy. We also say that G is (semi-) greedy if it is $ r $-(semi-) greedy for every choice of $ r $. The No-Cycle Lemma shows that every tree is greedy \cite{Bunde}.

\begin{figure}[ht]
	\centering
	\includegraphics[scale=0.8]{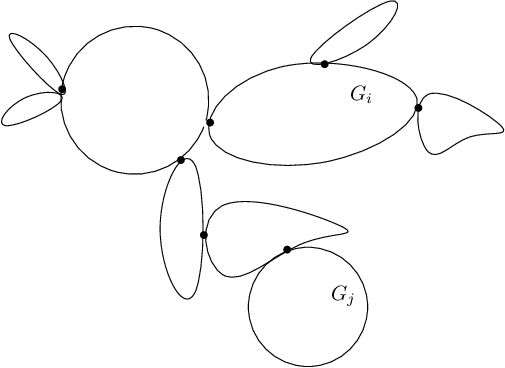}
	\caption{A polymer graph with monomer units $ G_1,...,G_k $ \cite{Alikhani4}.}\label{f8}
\end{figure}

The following theorem gives an upper bound for a polymer graph:

\begin{theorem}\label{***}
	If $ G $ is a polymer graph obtained by point-attaching from pairwise disjoint connected graphs $ G_1,...,G_k $, then $ \pi(G)\leqslant\displaystyle\prod_{i=1}^k\pi(G_i) $.
\end{theorem}
\begin{proof}
	Let $ G_1,...,G_k $ be a finite sequence of pairwise disjoint connected graphs. By the distance bound, we know that the maximum pebbling number of point-attaching from $ G_1,...,G_k $ is the chain of the graphs with composed of monomers $ \lbrace G_{i}\rbrace_{i=1}^k $ with respect to the nodes $ \lbrace x_i\rbrace_{i=1}^{k-1} $ (see Figure \ref{f9}(b)). Let $ C(G_1,...,G_k) $ be the chain of graphs and $ r $ be a target vertex in $ G_k $. Since a pebbling step from $ x_i $ to $ v\in G_{i+1} $ is r-greedy pebbling step in any $ r $-solvable configuration $ f $ on the graph $ C(G_1,...,G_k) $ and $ \pi_t(G)\leqslant t\pi(G) $, so the result follows inductively. 
	\qed
\end{proof}

We consider some particular cases of these graphs and study their upper bound for pebbling number. First we consider the link and chain of graphs. 
\begin{figure}[ht]
	\centering
	\includegraphics[scale=0.8]{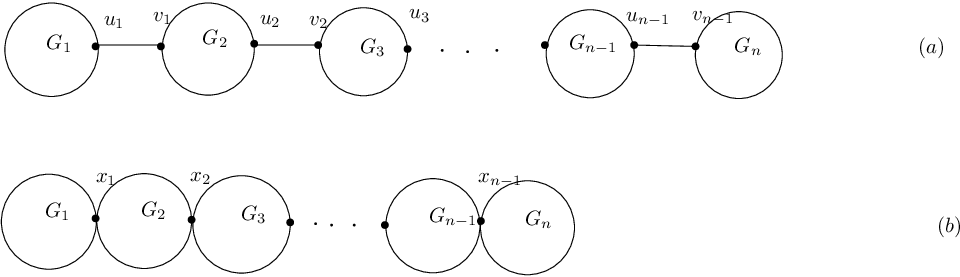}
	\caption{(a) Link of $ n $ graphs $ G_1,...,G_n $, (b) Chain of $ n $ graphs $ G_1,...,G_n $.}\label{f9}
\end{figure}

\begin{corollary}
	\begin{itemize}
		\item[(i)]
		Let $ G $ be a polymer graph with composed of monomers $ \lbrace G_{i}\rbrace_{i=1}^n $ with respect to the nodes $ \lbrace u_{i}, v_i\rbrace_{i=1}^{n-1} $. If $ G $ is the link of graphs (see Figure \ref{f9}), then
		\[
		\pi(G)\leqslant 2^{n-1}\displaystyle\prod_{i=1}^n\pi(G_i),
		\]
		this upper bound is sharp for the cactus graph $ Q_n+(n-1)e $.
		\item[(ii)]
		Let $ G $ be the chain of graphs with composed of monomers $ \lbrace G_{i}\rbrace_{i=1}^n $ with respect to the nodes $ \lbrace x_i\rbrace_{i=1}^{n-1} $. Then 
		\[
		\pi(G)\leqslant\displaystyle\prod_{i=1}^n\pi(G_i), 
		\]
		in particular, it is sharp for the cactus graph $ Q_n $.
	\end{itemize}
\end{corollary}

As another example of point-attaching graph, consider the graph $ K_n $ and $ n $ copies of $ K_m $. By definition, the graph $ Q(n,m) $ is obtained by identifying each vertex of $ K_n $ with a vertex of a unique $ K_m $. Actually $ Q(n,m)=K_n\circ K_{m-1} $, where $\circ$ is corona product. Note that the corona product $G\circ H$ of two graphs $G$ and $H$ is defined as the graph obtained by taking one copy of $G$ and $\vert V(G)\vert $ copies of $H$ and joining the $i$-th vertex of $G$ to every vertex in the $i$-th copy of $H$.

The following theorem gives the pebbling number and optimal pebbling number of corona of a complete graph with any arbitrary graph. 
\begin{theorem}{\rm \cite{Alikhani1}}\label{cor}
	For any graph $ H $ of order $ h $, $ \pi(K_n\circ H)=nh+2n+2 $, $ \pi^*(K_n\circ H)=4 $ ($ n>2 $).
\end{theorem}

Using Theorem \ref{cor}, we have the following corollary: 

\begin{corollary}
	For $ n>2 $, $ \pi(Q(n,m))=mn+n+2 $, $ \pi^*(Q(n,m))=4 $.
\end{corollary}

\begin{figure}[ht]
	\centering
	\includegraphics[scale=0.5]{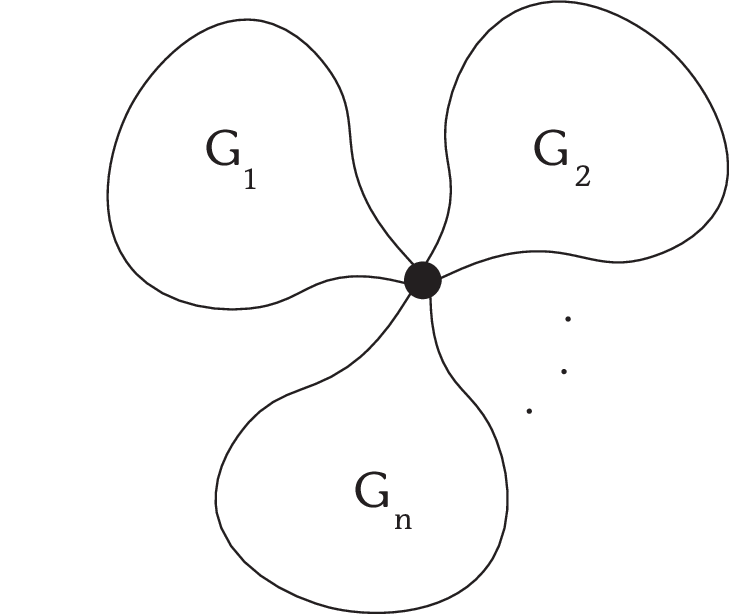}
	\caption{Bouquet of $ n $ graphs $ G_1,...,G_n $.}\label{f10}
\end{figure}

Here, we consider another kind of polymer graphs which is called the bouquet of 
graphs (Figure \ref{f10}). 

\begin{theorem} 
	Let $G_1,G_2, \ldots , G_n$ be a finite sequence of pairwise disjoint connected graphs and let
	$x_i \in V(G_i)$. Let $B(G_1,...,G_n)$ be the bouquet of graphs $\{G_i\}_{i=1}^n$ with respect to the vertices $\{x_i\}_{i=1}^n$ and obtained by identifying the vertex $x_i$ of the graph $G_i$ with node $x$ (see Figure \ref{f10}). Then,	
	\[
	\pi(B(G_1,...,G_n))\leqslant\pi(G_1)\pi(G_2)+\displaystyle\sum_{i=3}^n(\pi(G_i)-1),
	\]
	where $ \pi(G_1)\geqslant\pi(G_2)\geqslant ...\geqslant\pi(G_n) $ . In particular, this upper bound is tight for the generalized friendship graph $ F_{n,4} $.
\end{theorem}
\begin{proof}
	Let $G_1,G_2, \ldots , G_n$ be a finite sequence of pairwise disjoint connected graphs and let
	$x_i \in V(G_i)$. Let $G$ be the bouquet of graphs $\{G_i\}_{i=1}^n$ with respect to the vertices $\{x_i\}_{i=1}^n$ and obtained by identifying the vertex $x_i$ of the graph $G_i$ with node $x$. We have two cases:
	\begin{itemize}
		\item[(i)]
		The target vertex $ r $ is the node $ x $ in the the bouquet graph $ B(G_1,...,G_n) $.  In this case, by
		pigeonhole principle we have $ \pi(B,r)=\displaystyle\sum_{i=1}^n(\pi(G_i)-1)+1 $.
		\item[(ii)]
		Let $ r\neq x $ be an arbitrary target vertex in $ V(G_j) $.  If we put $ \pi(G_j) $ pebbles on the node $ x $, then by Theorem \ref{***}, we have $ \pi(B,r)\leqslant\pi(G_1)\pi(G_j)+\displaystyle\sum_{i\neq 1,j}^n(\pi(G_i)-1) $.
	\end{itemize}
	Therefore $ \pi(B)=\max_r\pi(B,r)\leqslant\pi(G_1)\pi(G_2)+\displaystyle\sum_{i=3}^n(\pi(G_i)-1) $.
	\qed
\end{proof}


\vspace*{2cm}


\end{document}